\documentclass[12pt]{article}

\usepackage[english]{babel}

\usepackage[letterpaper,top=2cm,bottom=2cm,left=3cm,right=3cm,marginparwidth=1.75cm]{geometry}

\usepackage{amsmath}
\usepackage{amssymb}
\usepackage{amsthm}
\usepackage{graphicx}
\usepackage[colorlinks=true, allcolors=blue]{hyperref}
\usepackage{longtable}
\usepackage[table]{xcolor}

\newtheorem{theorem}{Theorem}

\newtheorem{corollary}[theorem]{Corollary}
\newtheorem{proposition}[theorem]{Proposition}

\theoremstyle{definition}
\newtheorem{definition}[theorem]{Definition}

\title{Nearly tight weighted $2$-designs \\ in complex projective spaces of every dimension}
\author{John Jasper\thanks{Department of Mathematics and Statistics, Air Force Institute of Technology, Wright-Patterson AFB, Ohio, USA} \and Dustin G.\ Mixon\thanks{Department of Mathematics, The Ohio State University, Columbus, Ohio, USA} \thanks{Translational Data Analytics Institute, The Ohio State University, Columbus, Ohio, USA}}
\date{}

\begin{document}
\maketitle

\begin{abstract}
We use dense Sidon sets to construct small weighted projective $2$-designs.
This represents quantitative progress on Zauner's conjecture.
\end{abstract}

\section{Introduction}

In his PhD thesis~\cite{Zauner:99}, Zauner conjectured that for every $d\in\mathbb{N}$, there exists an arrangement of $d^2$ distinct points in $\mathbb{CP}^{d-1}$ with the property that every pair of points has the same Fubini--Study distance.
Such arrangements are known as \textit{symmetric informationally complete positive operator--valued measures} (or SICs) in quantum information theory, where they find use in quantum state tomography.

Over the last decade, there have been three main approaches to make progress on Zauner's conjecture.
First, computational investigations have produced numerical approximations of putative SICs (as well as some exact SICs) in finitely many dimensions~\cite{ScottG:10}.
Further analysis in~\cite{ApplebyFMY:17} then established that the coordinates of each of these exact SICs reside in an abelian extention of a particular real quadratic extension of $\mathbb{Q}$.
This discovery prompted a second approach to Zauner's conjecture, which leverages conjectures in algebraic number theory (such as those due to Stark) to find additional exact SICs and even a conditional proof of Zauner's conjecture~\cite{Kopp:21,ApplebyBGHM:22,BengtssonGM:24,ApplebyFK:25}.

As a third approach, \cite{PandeyEtal:20} reformulated Zauner's conjecture in terms of the entanglement breaking rank $n(d)$ of a certain quantum channel over $\mathbb{C}^{d\times d}$.

\begin{proposition}
\label{prop.bound of zauner}
For each $d\in\mathbb{N}$, it holds that $n(d)\geq d^2$, with equality precisely when there exists a SIC in $\mathbb{CP}^{d-1}$.
\end{proposition}

Importantly, this allows for quantitative progress on Zauner's conjecture:
For each $d\in\mathbb{N}$, find an upper bound on $n(d)$.
Then the closer this upper bound is to $d^2$, the ``closer'' we are to a proof of Zauner's conjecture in dimension $d$.
Soon after the release of~\cite{PandeyEtal:20}, a follow-up paper~\cite{IversonKM:21} showed that $n(d)$ equals the size of the smallest \textit{weighted $2$-design} for $\mathbb{CP}^{d-1}$, thereby identifying Proposition~\ref{prop.bound of zauner} above with Theorem~4 in~\cite{Scott:06}.

\begin{definition}
The unit vectors $x_1,\ldots,x_n\in\mathbb{C}^d$ are said to form a \textbf{weighted $2$-design for $\mathbb{CP}^{d-1}$} if there exist weights $w_1,\ldots,w_n\geq0$ such that the weighted combination
\[
\sum_{k=1}^n w_k(x_k^{\otimes 2})(x_k^{\otimes 2})^*
\]
equals the orthogonal projection onto the subspace of symmetric tensors in $(\mathbb{C}^d)^{\otimes 2}$.
\end{definition}

The following summarizes all of the best known upper bounds on $n(d)$:

\begin{proposition}\
\label{prop.known constructions}
\begin{itemize}
\item[(a)]
$n(d)=d^2$ whenever a SIC in $\mathbb{CP}^{d-1}$ is known to exist.
\item[(b)]
$n(d)\leq kd^2 + 2d$ whenever $kd + 1$ is a prime power with $k\in\mathbb{N}$.
\item[(c)]
$n(d)\leq d^2 + 1$ whenever $d-1$ is a prime power.
\item[(d)]
$n(d)\leq d^2 + d - 1$ whenever $d$ is a prime power.
\item[(e)]
$n(d)\leq\binom{d+1}{2}^2$.
\end{itemize}
\end{proposition}

\begin{proof}
First, (a) follows from the fact that SICs are weighted projective $2$-designs of minimal size. 
Next, (b) follows from combining Theorem~4.1 with Proposition~4.2 in~\cite{RoyS:07}.
Also, (c) and (d) follow from Corollaries~4.4 and~4.6 in~\cite{BodmannH:16}.
Finally, (e) follows from Corollary~7 in~\cite{IversonKM:21}.
\end{proof}

The weighted projective $2$-designs that imply Proposition~\ref{prop.known constructions}(c) and~(d) are instances of the same \textit{Bodmann--Haas construction}~\cite{BodmannH:16}.
In this paper, we identify a much larger class of weighted projective $2$-designs that arise from this construction, which in turn implies our main result:

\begin{theorem}
\label{thm.main}
$n(d)\leq d^2+O(d^{1.525})$.
\end{theorem}

This is the first known general upper bound on $n(d)$ that is $o(d^4)$, let alone sharp up to lower-order terms.
In the next section, we present the Bodmann--Haas construction and show that applying it to \textit{Sidon sets} results in weighted projective $2$-designs.
Next, Section~3 reviews the densest known Sidon sets and uses them to construct weighted projective $2$-designs that are \textit{nearly tight} (meaning their size is close to the lower bound $d^2$).
We conclude in Section~4 with a discussion.

\section{The Bodmann--Haas construction}

In what follows, we describe a construction technique due to Bodmann and Haas~\cite{BodmannH:16} that was originally obtained by generalizing a particular construction of mutually unbiased bases due to Godsil and Roy~\cite{GodsilR:09}.

\begin{definition}
Fix a finite abelian group $G$.
The \textbf{Bodmann--Haas construction} is a map that receives a subset $S\subseteq G$ and returns a sequence of $|G|+|S|$ unit vectors in $\mathbb{C}^S$.
In particular, each character $\alpha\in\hat{G}$ determines the vector $x_\alpha\in\mathbb{C}^S$ defined by
\[
x_\alpha(s)
=\frac{1}{\sqrt{|S|}}\alpha(s)
\qquad
(s\in S),
\]
while each $r\in S$ determines $e_r\in\mathbb{C}^S$ defined in terms of the Kronecker delta by
\[
e_r(s)=\delta_{r,s}
\qquad
(s\in S).
\]
Overall, given $S\subseteq G$, the Bodmann--Haas construction returns $\{x_\alpha\}_{\alpha\in\hat{G}}\cup\{e_r\}_{r\in S}$.
\end{definition}

Bodmann and Haas~\cite{BodmannH:16} used this to construct projective codes that achieve equality in the \textit{orthoplex bound}, and then seemingly as an afterthought, they verified that these codes are weighted projective $2$-designs.
In this section, we find many more instances in which the Bodmann--Haas construction returns a weighted projective $2$-design.

\begin{definition}
A \textbf{Sidon set} is a subset $S$ of a finite abelian group $G$ with the property that the map $\{a,b\}\mapsto a+b$ is injective over $a,b\in S$ (allowing $a=b$).
\end{definition}

\begin{theorem}
\label{thm.Sidon characterization}
Fix a finite abelian group $G$.
Then for any Sidon set $S\subseteq G$, the Bodmann--Haas construction applied to $S$ returns a weighted $2$-design for $\mathbb{CP}^{|S|-1}$.
\end{theorem}

\begin{proof}
Consider the transpose map $T$ over $(\mathbb{C}^S)^{\otimes 2}$ defined by taking $T(x\otimes y) = y\otimes x$ and extending linearly.
The orthogonal projection onto the subspace of symmetric tensors is then given by $P=\frac{1}{2}(I+T)$, where $I$ denotes the identity map.
Identifing $P$ with its matrix representation relative to the basis $\{e_s\otimes e_{s'}\}_{s,s'\in S}$, then the matrix entries are given by
\begin{align*}
P_{(s,s'),(t,t')}
&=(e_s\otimes e_{s'})^*\frac{1}{2}(I+T)(e_t\otimes e_{t'})\\
&=\frac{1}{2}(e_s^*\otimes e_{s'}^*)(e_t\otimes e_{t'}+e_{t'}\otimes e_t)
=\frac{1}{2}(\delta_{s,t}\delta_{s',t'}+\delta_{s,t'}\delta_{s',t}).
\end{align*}
This leads us to consider the matrices $A$ and $B$ defined by
\begin{align*}
A_{(s,s'),(t,t')}
&=\left\{
\begin{array}{cl}
1&\text{if $s\neq s'$, $t\neq t'$, $\{s,s'\}=\{t,t'\}$}\\
0&\text{else}
\end{array}
\right\},\\
B_{(s,s'),(t,t')}
&=\left\{
\begin{array}{cl}
1&\text{if $s=s'$, $t=t'$, $\{s,s'\}=\{t,t'\}$}\\
0&\text{else}
\end{array}
\right\},
\end{align*}
since then $P=\frac{1}{2}A+B$.

Given a Sidon set $S\subseteq G$, the Bodmann--Haas construction returns the sequence $\{x_\alpha\}_{\alpha\in \hat{G}}\cup\{e_r\}_{r\in S}$.
Consider the matrices
\[
X:=\sum_{\alpha\in\hat{G}}(x_\alpha^{\otimes 2})(x_\alpha^{\otimes 2})^*,
\qquad
E:=\sum_{r\in S}(e_r^{\otimes 2})(e_r^{\otimes 2})^*.
\]
We will express $P$ as a nonnegative combination of $X$ and $E$, thereby demonstrating that $\{x_\alpha\}_{\alpha\in \hat{G}}\cup\{e_r\}_{r\in S}$ is a weighted projective $2$-design.
We proceed by computing matrix entries:
\begin{align*}
X_{(s,s'),(t,t')}
&=(e_s\otimes e_{s'})^*X(e_t\otimes e_{t'})\\
&=\sum_{\alpha\in\hat{G}}(e_s\otimes e_{s'})^*(x_\alpha^{\otimes 2})(x_\alpha^{\otimes 2})^*(e_t\otimes e_{t'})\\
&=\sum_{\alpha\in\hat{G}}(e_s^*x_\alpha)(e_{s'}^*x_\alpha)(x_\alpha^*e_t)(x_\alpha^*e_{t'})\\
&=\frac{1}{|S|^2}\sum_{\alpha\in\hat{G}}\alpha(s)\alpha(s')\overline{\alpha(t)\alpha(t')}\\
&=\frac{1}{|S|^2}\sum_{\alpha\in\hat{G}}\alpha(s+s'-t-t')
=\left\{\begin{array}{cl}
\frac{|G|}{|S|^2}&\text{if $s+s'=t+t'$}\\
0&\text{else}
\end{array}
\right\}.
\end{align*}
In particular, since $S$ is Sidon by assumption, we have $X=\frac{|G|}{|S|^2}(A+B)$.
Next,
\begin{align*}
E_{(s,s'),(t,t')}
&=(e_s\otimes e_{s'})^*E(e_t\otimes e_{t'})\\
&=\sum_{r\in S}(e_s\otimes e_{s'})^*(e_r^{\otimes 2})(e_r^{\otimes 2})^*(e_t\otimes e_{t'})\\
&=\sum_{r\in S}(e_s^*e_r)(e_{s'}^*e_r)(e_r^*e_t)(e_r^*e_{t'})\\
&=\sum_{r\in S}\delta_{s,r}\delta_{s',r}\delta_{r,t}\delta_{r,t'}
=\left\{\begin{array}{cl}
1&\text{if $s=s'=t=t'$}\\
0&\text{else}
\end{array}\right\},
\end{align*}
and so $E=B$.
Overall, we have
\[
P
=\frac{1}{2}A+B
=\frac{|S|^2}{2|G|}X+\frac{1}{2}E,
\]
and so $\{x_\alpha\}_{\alpha\in \hat{G}}\cup\{e_r\}_{r\in S}$ is a weighted projective $2$-design with \(w_{\alpha} = \frac{|S|^2}{2|G|}\) for \(\alpha\in\hat{G}\) and \(w_{r}=\frac{1}{2}\) for \(r\in S\).
\end{proof}

\section{Nearly tight weighted projective $2$-designs}

Recall that we seek a sharp upper bound on $n(d)$.
Letting $m(d)$ denote the size of the smallest group that contains a Sidon set of size $d$, then the following bound is an immediate consequence of Theorem~\ref{thm.Sidon characterization}:

\begin{corollary}
\label{cor.bound from sidon}
$n(d)\leq m(d)+d$.
\end{corollary}

As we will soon see, Corollary~\ref{cor.bound from sidon} is the sharpest known upper bound on $n(d)$ for all but finitely many $d\in\mathbb{N}$.
To obtain a tight upper bound on $m(d)$, we seek \textit{dense} Sidon sets, i.e., large Sidon sets relative to their parent groups.
Note that the equivalence
\[
a+b = c+d
\qquad
\Longleftrightarrow
\qquad
a-d = c-b
\]
implies that $S\subseteq G$ is a Sidon set precisely when all pairwise differences between distinct members of $S$ are different.
The pigeonhole principle then gives the necessary condition
\begin{equation}
\label{eq.bound on sidon sets}
|S|\big(|S|-1\big)
\leq |G|-1,
\end{equation}
i.e., $|S|\leq(1+o(1))\sqrt{|G|}$.
Meanwhile, every known infinite family of Sidon sets that satisfies $|S|\geq(1-o(1))\sqrt{|G|}$ stems from the following constructions; see~\cite{EberhardM:23} and references therein:

\begin{proposition}
\label{prop.known dense sidon sets}
In what follows,\footnote{For convenience, we report the parameters $(|G|,|S|)$ of each construction in terms of $q$.} $q$ is a prime power and $S$ is a Sidon set in $G$.
\begin{itemize}
\item[(a)]
\textbf{Erd\H{o}s--Tur\'{a}n.}
$(q^2,q)$.
$G=(\mathbb{F}_q,+)^2$, 
$S=\{(x,x^2):x\in\mathbb{F}_q\}$, 
$\operatorname{char}(q)>2$.
\item[(b)]
\textbf{Singer.}
$(q^2+q+1,q+1)$.
$G=\mathbb{F}_{q^3}^\times/\mathbb{F}_q^{\times}$,
$S=\{[x]:x\in\mathbb{F}_{q^3}^\times,\operatorname{tr}x=0\}$.
\item[(c)]
\textbf{Bose.}
$(q^2-1,q)$.
$G=\mathbb{F}_{q^2}^\times$,
$S=\{x\in \mathbb{F}_{q^2}^\times:\operatorname{tr}x=0\}$.
\item[(d)]
\textbf{Spence.}
$(q(q-1),q-1)$.
$G=\mathbb{F}_q^\times\times(\mathbb{F}_q,+)$,
$S=\{(x,x):x\in\mathbb{F}_q^\times\}$.
\item[(e)]
\textbf{Hughes.}
$((q-1)^2,q-2)$.
$G=(\mathbb{F}_q^\times)^2$,
$S=\{(x,y):x,y\neq0,x+y=1\}$.
\end{itemize}
\end{proposition}

Some of these Sidon sets have already met the Bodmann--Haas construction.
First, Godsil and Roy~\cite{GodsilR:09} used a relative difference set isomorphic to the Erd\"{o}s--Tur\'{a}n Sidon set to construct mutually unbiased bases.
Later, Bodmann and Haas~\cite{BodmannH:16} used the Singer Sidon set as well as a relative difference set isomorphic to the Bose Sidon set to construct projective codes and designs.
We are not aware of the Spence or Hughes Sidon sets appearing previously in the literature on codes and designs.

Since subsets of Sidon sets are also Sidon sets, we may bound $m(d)$ by the size of the smallest group described in Proposition~\ref{prop.known dense sidon sets} whose Sidon set has size at least $d$.
This allows us to prove our main result:

\begin{proof}[Proof of Theorem~\ref{thm.main}]
It suffices to demonstrate $m(d)\leq d^2+O(d^{1.525})$, since then the result follows from Corollary~\ref{cor.bound from sidon}.
Considering Proposition~\ref{prop.known dense sidon sets}(a), it holds that $m(d)$ is at most $p(d)^2$, where $p(d)$ denotes the smallest prime $\geq d$.
Finally, the main result in~\cite{BakerHP:01} gives $p(d)\leq d+O(d^{0.525})$, and so we are done.
\end{proof}

We conclude this section by discussing Table~\ref{table}, which compares our bound based on Corollary~\ref{cor.bound from sidon} and Proposition~\ref{prop.known dense sidon sets} to the previous bounds in Proposition~\ref{prop.known constructions}.
Implementing Proposition~\ref{prop.known constructions}(a) is cumbersome since the known SICs are not maintained in a public database.
We first collected the dimensions listed in~\cite{Flammia:online}, which represents a complete survey of known dimensions as of September 2017.
Then we included the dimensions $23$ (due to~\cite{Kopp:21}), $52$ (due to~\cite{BengtssonGM:24}), and $67$ and $103$ (due to~\cite{ApplebyBGHM:22}).
As far as we know, this gives all dimensions $\leq150$ for which an exact SIC has been published.
We do not include any of the unpublished exact SICs that were announced in~\cite{Kopp:21,Grassl:21a,Grassl:21b}.

In Table~\ref{table}, the $d^2$ column gives the best known lower bound on $n(d)$, while the next two columns give competing upper bounds on $n(d)$.
We highlight the better of the two upper bounds in yellow, and when this matches the lower bound, we also highlight the $d^2$ column.
At times, the Sidon set we use is obtained by removing any $k$ points from the $\operatorname{Hughes}(q)$ Sidon set, which we denote by $\operatorname{Hughes}(q)-k$.

We make a few observations from Table~\ref{table}.
For every $d\leq 150$, exactly one of three things happens: either a SIC exists, or the upper bounds tie, or our upper bound is strictly better by using (a subset of) the Hughes Sidon set.
In this regime, none of the bounds from Proposition~\ref{prop.known constructions} make use of part (e).
Similarly, $m(d)$ is never achieved by the Erd\H{o}s-Tur\'{a}n Sidon set since there is always Bose Sidon set of the same size but in a smaller group.
When the upper bounds tie, it's frequently because our use of the Singer and Bose Sidon sets align with Proposition~\ref{prop.known constructions}(c) and (d), respectively, as these correspond to the original application of the Bodmann--Haas construction.
However, the upper bounds also tie between Proposition~\ref{prop.known constructions}(b) with $k=1$ and our use of Spence Sidon sets.
In hindsight, these ties are made possible because the parameters match, but we don't know of a deeper relationship between these weighted projective $2$-designs.
One of the main takeaways from Table~\ref{table} is that the improvements we provide are due to (subsets of) the Hughes Sidon sets.

\section{Discussion}

In this paper, we used Sidon sets to construct new weighted projective $2$-designs, which in turn represents quantitative progress towards Zauner's conjecture.
In this section, we highlight some fundamental limits of our approach.

Note that the necessary condition \eqref{eq.bound on sidon sets} gives that any group containing a Sidon set of size $d$ must have cardinality at least $d^2-d+1$.
The weighted projective $2$-design resulting from the Bodmann--Haas construction then has size at least $d^2+1$.
As such, our approach is not powerful enough to establish Zauner's conjecture that $n(d)=d^2$.
We suspect that the easiest way to improve the bounds on $n(d)$ in Table~\ref{table} is to find more exact SICs, of which several have been announced in~\cite{Kopp:21,Grassl:21a,Grassl:21b}.

Short of a proof of Zauner's conjecture, it would be interesting to improve Theorem~\ref{thm.main}.
To this end, the Bodmann--Haas construction is somewhat limiting, since we believe $m(d)$ is at least nearly achieved by subsets of the Sidon sets in Proposition~\ref{prop.known dense sidon sets}.
As such, any improvement must come from better bounds on gaps between primes.
Heuristics suggest that the best possible bound is given by Cram\'{e}r's conjecture, and so an estimate of the form $n(d)\leq d^2+o(d\log^2d)$ would likely require a different approach.

\section*{Acknowledgments}

The authors thank Emily King for a helpful discussion of relevant literature.
JJ was supported by NSF DMS 2220320. 
DGM was supported by NSF DMS 2220304.
The views expressed are those of the authors and do not reflect the official guidance or position of the United States Government, the Department of Defense, the United States Air Force,
or the United States Space Force.

\newpage

\begin{small}
\begin{longtable}{|c|ccc|ccc|}
\caption{\label{table}Best known bounds on $n(d)$}
\\\hline
\phantom{a}$d$\phantom{a}	& \phantom{aaaa}$d^2$\phantom{aaaa} & \phantom{a}$\begin{array}{c}\text{Prior}\\\text{bound}\\ \text{on $n(d)$}\end{array}$\phantom{a} & \!\!\!$\begin{array}{c}\text{Best known}\\\text{bound on}\\\text{$m(d)+d$}\end{array}$\!\!\! & $\begin{array}{c}\text{Prop.~\ref{prop.known constructions}}\\\text{part}\end{array}$ & Sidon set	& Group \\\hline \endhead 
\hline\endfoot
$1$	& \cellcolor{yellow}$1$	& \cellcolor{yellow}$1$	& $2$	& (a)	& $\{0\}$	& $\{0\}$ \\
$2$	& \cellcolor{yellow}$4$	& \cellcolor{yellow}$4$	& $5$	& (a)	& $\operatorname{Bose}(2)$	& $\mathbb{Z}_{3}$\\
$3$	& \cellcolor{yellow}$9$	& \cellcolor{yellow}$9$	& $10$	& (a)	& $\operatorname{Singer}(2)$	& $\mathbb{Z}_{7}$\\
$4$	& \cellcolor{yellow}$16$	& \cellcolor{yellow}$16$	& $17$	& (a)	& $\operatorname{Singer}(3)$	& $\mathbb{Z}_{13}$\\
$5$	& \cellcolor{yellow}$25$	& \cellcolor{yellow}$25$	& $26$	& (a)	& $\operatorname{Singer}(4)$	& $\mathbb{Z}_{21}$\\
$6$	& \cellcolor{yellow}$36$	& \cellcolor{yellow}$36$	& $37$	& (a)	& $\operatorname{Singer}(5)$	& $\mathbb{Z}_{31}$\\
$7$	& \cellcolor{yellow}$49$	& \cellcolor{yellow}$49$	& $55$	& (a)	& $\operatorname{Bose}(7)$	& $\mathbb{Z}_{48}$\\
$8$	& \cellcolor{yellow}$64$	& \cellcolor{yellow}$64$	& $65$	& (a)	& $\operatorname{Singer}(7)$	& $\mathbb{Z}_{57}$\\
$9$	& \cellcolor{yellow}$81$	& \cellcolor{yellow}$81$	& $82$	& (a)	& $\operatorname{Singer}(8)$	& $\mathbb{Z}_{73}$\\
$10$	& \cellcolor{yellow}$100$	& \cellcolor{yellow}$100$	& $101$	& (a)	& $\operatorname{Singer}(9)$	& $\mathbb{Z}_{91}$\\
$11$	& \cellcolor{yellow}$121$	& \cellcolor{yellow}$121$	& $131$	& (a)	& $\operatorname{Bose}(11)$	& $\mathbb{Z}_{120}$\\
$12$	& \cellcolor{yellow}$144$	& \cellcolor{yellow}$144$	& $145$	& (a)	& $\operatorname{Singer}(11)$	& $\mathbb{Z}_{133}$\\
$13$	& \cellcolor{yellow}$169$	& \cellcolor{yellow}$169$	& $181$	& (a)	& $\operatorname{Bose}(13)$	& $\mathbb{Z}_{168}$\\
$14$	& \cellcolor{yellow}$196$	& \cellcolor{yellow}$196$	& $197$	& (a)	& $\operatorname{Singer}(13)$	& $\mathbb{Z}_{183}$\\
$15$	& \cellcolor{yellow}$225$	& \cellcolor{yellow}$225$	& $255$	& (a)	& $\operatorname{Spence}(16)$	& $\mathbb{Z}_{15}\times \mathbb{Z}_{2}^{4}$\\
$16$	& \cellcolor{yellow}$256$	& \cellcolor{yellow}$256$	& $271$	& (a)	& $\operatorname{Bose}(16)$	& $\mathbb{Z}_{255}$\\
$17$	& \cellcolor{yellow}$289$	& \cellcolor{yellow}$289$	& $290$	& (a)	& $\operatorname{Singer}(16)$	& $\mathbb{Z}_{273}$\\
$18$	& \cellcolor{yellow}$324$	& \cellcolor{yellow}$324$	& $325$	& (a)	& $\operatorname{Singer}(17)$	& $\mathbb{Z}_{307}$\\
$19$	& \cellcolor{yellow}$361$	& \cellcolor{yellow}$361$	& $379$	& (a)	& $\operatorname{Bose}(19)$	& $\mathbb{Z}_{360}$\\
$20$	& \cellcolor{yellow}$400$	& \cellcolor{yellow}$400$	& $401$	& (a)	& $\operatorname{Singer}(19)$	& $\mathbb{Z}_{381}$\\
$21$	& \cellcolor{yellow}$441$	& \cellcolor{yellow}$441$	& $505$	& (a)	& $\operatorname{Hughes}(23)$	& $\mathbb{Z}_{22}^2$\\
$22$	& $484$	& \cellcolor{yellow}$528$	& \cellcolor{yellow}$528$	& (b)	& $\operatorname{Spence}(23)$	& $\mathbb{Z}_{506}$\\
$23$	& \cellcolor{yellow}$529$	& \cellcolor{yellow}$529$	& $551$	& (a)	& $\operatorname{Bose}(23)$	& $\mathbb{Z}_{528}$\\
$24$	& \cellcolor{yellow}$576$	& \cellcolor{yellow}$576$	& $577$	& (a)	& $\operatorname{Singer}(23)$	& $\mathbb{Z}_{553}$\\
$25$	& $625$	& \cellcolor{yellow}$649$	& \cellcolor{yellow}$649$	& (d)	& $\operatorname{Bose}(25)$	& $\mathbb{Z}_{624}$\\
$26$	& $676$	& \cellcolor{yellow}$677$	& \cellcolor{yellow}$677$	& (c)	& $\operatorname{Singer}(25)$	& $\mathbb{Z}_{651}$\\
$27$	& $729$	& \cellcolor{yellow}$755$	& \cellcolor{yellow}$755$	& (d)	& $\operatorname{Bose}(27)$	& $\mathbb{Z}_{728}$\\
$28$	& \cellcolor{yellow}$784$	& \cellcolor{yellow}$784$	& $785$	& (a)	& $\operatorname{Singer}(27)$	& $\mathbb{Z}_{757}$\\
$29$	& $841$	& \cellcolor{yellow}$869$	& \cellcolor{yellow}$869$	& (d)	& $\operatorname{Bose}(29)$	& $\mathbb{Z}_{840}$\\
$30$	& \cellcolor{yellow}$900$	& \cellcolor{yellow}$900$	& $901$	& (a)	& $\operatorname{Singer}(29)$	& $\mathbb{Z}_{871}$\\
$31$	& \cellcolor{yellow}$961$	& \cellcolor{yellow}$961$	& $991$	& (a)	& $\operatorname{Bose}(31)$	& $\mathbb{Z}_{960}$\\
$32$	& $1024$	& \cellcolor{yellow}$1025$	& \cellcolor{yellow}$1025$	& (c)	& $\operatorname{Singer}(31)$	& $\mathbb{Z}_{993}$\\
$33$	& $1089$	& \cellcolor{yellow}$1090$	& \cellcolor{yellow}$1090$	& (c)	& $\operatorname{Singer}(32)$	& $\mathbb{Z}_{1057}$\\
$34$	& $1156$	& $3536$	& \cellcolor{yellow}$1330$	& (b)	& $\operatorname{Hughes}(37)-1$	& $\mathbb{Z}_{36}^2$\\
$35$	& \cellcolor{yellow}$1225$	& \cellcolor{yellow}$1225$	& $1331$	& (a)	& $\operatorname{Hughes}(37)$	& $\mathbb{Z}_{36}^2$\\
$36$	& $1296$	& \cellcolor{yellow}$1368$	& \cellcolor{yellow}$1368$	& (b)	& $\operatorname{Spence}(37)$	& $\mathbb{Z}_{1332}$\\
$37$	& \cellcolor{yellow}$1369$	& \cellcolor{yellow}$1369$	& $1405$	& (a)	& $\operatorname{Bose}(37)$	& $\mathbb{Z}_{1368}$\\
$38$	& $1444$	& \cellcolor{yellow}$1445$	& \cellcolor{yellow}$1445$	& (c)	& $\operatorname{Singer}(37)$	& $\mathbb{Z}_{1407}$\\
$39$	& \cellcolor{yellow}$1521$	& \cellcolor{yellow}$1521$	& $1639$	& (a)	& $\operatorname{Hughes}(41)$	& $\mathbb{Z}_{40}^2$\\
$40$	& $1600$	& \cellcolor{yellow}$1680$	& \cellcolor{yellow}$1680$	& (b)	& $\operatorname{Spence}(41)$	& $\mathbb{Z}_{1640}$\\
$41$	& $1681$	& \cellcolor{yellow}$1721$	& \cellcolor{yellow}$1721$	& (d)	& $\operatorname{Bose}(41)$	& $\mathbb{Z}_{1680}$\\
$42$	& $1764$	& \cellcolor{yellow}$1765$	& \cellcolor{yellow}$1765$	& (c)	& $\operatorname{Singer}(41)$	& $\mathbb{Z}_{1723}$\\
$43$	& \cellcolor{yellow}$1849$	& \cellcolor{yellow}$1849$	& $1891$	& (a)	& $\operatorname{Bose}(43)$	& $\mathbb{Z}_{1848}$\\
$44$	& $1936$	& \cellcolor{yellow}$1937$	& \cellcolor{yellow}$1937$	& (c)	& $\operatorname{Singer}(43)$	& $\mathbb{Z}_{1893}$\\
$45$	& $2025$	& $8190$	& \cellcolor{yellow}$2161$	& (b)	& $\operatorname{Hughes}(47)$	& $\mathbb{Z}_{46}^2$\\
$46$	& $2116$	& \cellcolor{yellow}$2208$	& \cellcolor{yellow}$2208$	& (b)	& $\operatorname{Spence}(47)$	& $\mathbb{Z}_{2162}$\\
$47$	& $2209$	& \cellcolor{yellow}$2255$	& \cellcolor{yellow}$2255$	& (d)	& $\operatorname{Bose}(47)$	& $\mathbb{Z}_{2208}$\\
$48$	& \cellcolor{yellow}$2304$	& \cellcolor{yellow}$2304$	& $2305$	& (a)	& $\operatorname{Singer}(47)$	& $\mathbb{Z}_{2257}$\\
$49$	& $2401$	& \cellcolor{yellow}$2449$	& \cellcolor{yellow}$2449$	& (d)	& $\operatorname{Bose}(49)$	& $\mathbb{Z}_{2400}$\\
$50$	& $2500$	& \cellcolor{yellow}$2501$	& \cellcolor{yellow}$2501$	& (c)	& $\operatorname{Singer}(49)$	& $\mathbb{Z}_{2451}$\\
$51$	& $2601$	& $5304$	& \cellcolor{yellow}$2755$	& (b)	& $\operatorname{Hughes}(53)$	& $\mathbb{Z}_{52}^2$\\
$52$	& \cellcolor{yellow}$2704$	& \cellcolor{yellow}$2704$	& $2808$	& (a)	& $\operatorname{Spence}(53)$	& $\mathbb{Z}_{2756}$\\
$53$	& $2809$	& \cellcolor{yellow}$2861$	& \cellcolor{yellow}$2861$	& (d)	& $\operatorname{Bose}(53)$	& $\mathbb{Z}_{2808}$\\
$54$	& $2916$	& \cellcolor{yellow}$2917$	& \cellcolor{yellow}$2917$	& (c)	& $\operatorname{Singer}(53)$	& $\mathbb{Z}_{2863}$\\
$55$	& $3025$	& $18260$	& \cellcolor{yellow}$3419$	& (b)	& $\operatorname{Hughes}(59)-2$	& $\mathbb{Z}_{58}^2$\\
$56$	& $3136$	& $6384$	& \cellcolor{yellow}$3420$	& (b)	& $\operatorname{Hughes}(59)-1$	& $\mathbb{Z}_{58}^2$\\
$57$	& $3249$	& $13110$	& \cellcolor{yellow}$3421$	& (b)	& $\operatorname{Hughes}(59)$	& $\mathbb{Z}_{58}^2$\\
$58$	& $3364$	& \cellcolor{yellow}$3480$	& \cellcolor{yellow}$3480$	& (b)	& $\operatorname{Spence}(59)$	& $\mathbb{Z}_{3422}$\\
$59$	& $3481$	& \cellcolor{yellow}$3539$	& \cellcolor{yellow}$3539$	& (d)	& $\operatorname{Bose}(59)$	& $\mathbb{Z}_{3480}$\\
$60$	& $3600$	& \cellcolor{yellow}$3601$	& \cellcolor{yellow}$3601$	& (c)	& $\operatorname{Singer}(59)$	& $\mathbb{Z}_{3541}$\\
$61$	& $3721$	& \cellcolor{yellow}$3781$	& \cellcolor{yellow}$3781$	& (d)	& $\operatorname{Bose}(61)$	& $\mathbb{Z}_{3720}$\\
$62$	& $3844$	& \cellcolor{yellow}$3845$	& \cellcolor{yellow}$3845$	& (c)	& $\operatorname{Singer}(61)$	& $\mathbb{Z}_{3783}$\\
$63$	& $3969$	& \cellcolor{yellow}$4095$	& \cellcolor{yellow}$4095$	& (b)	& $\operatorname{Spence}(64)$	& $\mathbb{Z}_{63}\times \mathbb{Z}_{2}^{6}$\\
$64$	& $4096$	& \cellcolor{yellow}$4159$	& \cellcolor{yellow}$4159$	& (d)	& $\operatorname{Bose}(64)$	& $\mathbb{Z}_{4095}$\\
$65$	& $4225$	& \cellcolor{yellow}$4226$	& \cellcolor{yellow}$4226$	& (c)	& $\operatorname{Singer}(64)$	& $\mathbb{Z}_{4161}$\\
$66$	& $4356$	& \cellcolor{yellow}$4488$	& \cellcolor{yellow}$4488$	& (b)	& $\operatorname{Spence}(67)$	& $\mathbb{Z}_{4422}$\\
$67$	& \cellcolor{yellow}$4489$	& \cellcolor{yellow}$4489$	& $4555$	& (a)	& $\operatorname{Bose}(67)$	& $\mathbb{Z}_{4488}$\\
$68$	& $4624$	& \cellcolor{yellow}$4625$	& \cellcolor{yellow}$4625$	& (c)	& $\operatorname{Singer}(67)$	& $\mathbb{Z}_{4557}$\\
$69$	& $4761$	& $9660$	& \cellcolor{yellow}$4969$	& (b)	& $\operatorname{Hughes}(71)$	& $\mathbb{Z}_{70}^2$\\
$70$	& $4900$	& \cellcolor{yellow}$5040$	& \cellcolor{yellow}$5040$	& (b)	& $\operatorname{Spence}(71)$	& $\mathbb{Z}_{4970}$\\
$71$	& $5041$	& \cellcolor{yellow}$5111$	& \cellcolor{yellow}$5111$	& (d)	& $\operatorname{Bose}(71)$	& $\mathbb{Z}_{5040}$\\
$72$	& $5184$	& \cellcolor{yellow}$5185$	& \cellcolor{yellow}$5185$	& (c)	& $\operatorname{Singer}(71)$	& $\mathbb{Z}_{5113}$\\
$73$	& $5329$	& \cellcolor{yellow}$5401$	& \cellcolor{yellow}$5401$	& (d)	& $\operatorname{Bose}(73)$	& $\mathbb{Z}_{5328}$\\
$74$	& $5476$	& \cellcolor{yellow}$5477$	& \cellcolor{yellow}$5477$	& (c)	& $\operatorname{Singer}(73)$	& $\mathbb{Z}_{5403}$\\
$75$	& $5625$	& $11400$	& \cellcolor{yellow}$6159$	& (b)	& $\operatorname{Hughes}(79)-2$	& $\mathbb{Z}_{78}^2$\\
$76$	& $5776$	& $17480$	& \cellcolor{yellow}$6160$	& (b)	& $\operatorname{Hughes}(79)-1$	& $\mathbb{Z}_{78}^2$\\
$77$	& $5929$	& $35728$	& \cellcolor{yellow}$6161$	& (b)	& $\operatorname{Hughes}(79)$	& $\mathbb{Z}_{78}^2$\\
$78$	& $6084$	& \cellcolor{yellow}$6240$	& \cellcolor{yellow}$6240$	& (b)	& $\operatorname{Spence}(79)$	& $\mathbb{Z}_{6162}$\\
$79$	& $6241$	& \cellcolor{yellow}$6319$	& \cellcolor{yellow}$6319$	& (d)	& $\operatorname{Bose}(79)$	& $\mathbb{Z}_{6240}$\\
$80$	& $6400$	& \cellcolor{yellow}$6401$	& \cellcolor{yellow}$6401$	& (c)	& $\operatorname{Singer}(79)$	& $\mathbb{Z}_{6321}$\\
$81$	& $6561$	& \cellcolor{yellow}$6641$	& \cellcolor{yellow}$6641$	& (d)	& $\operatorname{Bose}(81)$	& $\mathbb{Z}_{6560}$\\
$82$	& $6724$	& \cellcolor{yellow}$6725$	& \cellcolor{yellow}$6725$	& (c)	& $\operatorname{Singer}(81)$	& $\mathbb{Z}_{6643}$\\
$83$	& $6889$	& \cellcolor{yellow}$6971$	& \cellcolor{yellow}$6971$	& (d)	& $\operatorname{Bose}(83)$	& $\mathbb{Z}_{6888}$\\
$84$	& $7056$	& \cellcolor{yellow}$7057$	& \cellcolor{yellow}$7057$	& (c)	& $\operatorname{Singer}(83)$	& $\mathbb{Z}_{6973}$\\
$85$	& $7225$	& $21845$	& \cellcolor{yellow}$7829$	& (b)	& $\operatorname{Hughes}(89)-2$	& $\mathbb{Z}_{88}^2$\\
$86$	& $7396$	& $14964$	& \cellcolor{yellow}$7830$	& (b)	& $\operatorname{Hughes}(89)-1$	& $\mathbb{Z}_{88}^2$\\
$87$	& $7569$	& $30450$	& \cellcolor{yellow}$7831$	& (b)	& $\operatorname{Hughes}(89)$	& $\mathbb{Z}_{88}^2$\\
$88$	& $7744$	& \cellcolor{yellow}$7920$	& \cellcolor{yellow}$7920$	& (b)	& $\operatorname{Spence}(89)$	& $\mathbb{Z}_{7832}$\\
$89$	& $7921$	& \cellcolor{yellow}$8009$	& \cellcolor{yellow}$8009$	& (d)	& $\operatorname{Bose}(89)$	& $\mathbb{Z}_{7920}$\\
$90$	& $8100$	& \cellcolor{yellow}$8101$	& \cellcolor{yellow}$8101$	& (c)	& $\operatorname{Singer}(89)$	& $\mathbb{Z}_{8011}$\\
$91$	& $8281$	& $49868$	& \cellcolor{yellow}$9307$	& (b)	& $\operatorname{Hughes}(97)-4$	& $\mathbb{Z}_{96}^2$\\
$92$	& $8464$	& $25576$	& \cellcolor{yellow}$9308$	& (b)	& $\operatorname{Hughes}(97)-3$	& $\mathbb{Z}_{96}^2$\\
$93$	& $8649$	& $34782$	& \cellcolor{yellow}$9309$	& (b)	& $\operatorname{Hughes}(97)-2$	& $\mathbb{Z}_{96}^2$\\
$94$	& $8836$	& $26696$	& \cellcolor{yellow}$9310$	& (b)	& $\operatorname{Hughes}(97)-1$	& $\mathbb{Z}_{96}^2$\\
$95$	& $9025$	& $18240$	& \cellcolor{yellow}$9311$	& (b)	& $\operatorname{Hughes}(97)$	& $\mathbb{Z}_{96}^2$\\
$96$	& $9216$	& \cellcolor{yellow}$9408$	& \cellcolor{yellow}$9408$	& (b)	& $\operatorname{Spence}(97)$	& $\mathbb{Z}_{9312}$\\
$97$	& $9409$	& \cellcolor{yellow}$9505$	& \cellcolor{yellow}$9505$	& (d)	& $\operatorname{Bose}(97)$	& $\mathbb{Z}_{9408}$\\
$98$	& $9604$	& \cellcolor{yellow}$9605$	& \cellcolor{yellow}$9605$	& (c)	& $\operatorname{Singer}(97)$	& $\mathbb{Z}_{9507}$\\
$99$	& $9801$	& $19800$	& \cellcolor{yellow}$10099$	& (b)	& $\operatorname{Hughes}(101)$	& $\mathbb{Z}_{100}^2$\\
$100$	& $10000$	& \cellcolor{yellow}$10200$	& \cellcolor{yellow}$10200$	& (b)	& $\operatorname{Spence}(101)$	& $\mathbb{Z}_{10100}$\\
$101$	& $10201$	& \cellcolor{yellow}$10301$	& \cellcolor{yellow}$10301$	& (d)	& $\operatorname{Bose}(101)$	& $\mathbb{Z}_{10200}$\\
$102$	& $10404$	& \cellcolor{yellow}$10405$	& \cellcolor{yellow}$10405$	& (c)	& $\operatorname{Singer}(101)$	& $\mathbb{Z}_{10303}$\\
$103$	& \cellcolor{yellow}$10609$	& \cellcolor{yellow}$10609$	& $10711$	& (a)	& $\operatorname{Bose}(103)$	& $\mathbb{Z}_{10608}$\\
$104$	& $10816$	& \cellcolor{yellow}$10817$	& \cellcolor{yellow}$10817$	& (c)	& $\operatorname{Singer}(103)$	& $\mathbb{Z}_{10713}$\\
$105$	& $11025$	& $22260$	& \cellcolor{yellow}$11341$	& (b)	& $\operatorname{Hughes}(107)$	& $\mathbb{Z}_{106}^2$\\
$106$	& $11236$	& \cellcolor{yellow}$11448$	& \cellcolor{yellow}$11448$	& (b)	& $\operatorname{Spence}(107)$	& $\mathbb{Z}_{11342}$\\
$107$	& $11449$	& \cellcolor{yellow}$11555$	& \cellcolor{yellow}$11555$	& (d)	& $\operatorname{Bose}(107)$	& $\mathbb{Z}_{11448}$\\
$108$	& $11664$	& \cellcolor{yellow}$11665$	& \cellcolor{yellow}$11665$	& (c)	& $\operatorname{Singer}(107)$	& $\mathbb{Z}_{11557}$\\
$109$	& $11881$	& \cellcolor{yellow}$11989$	& \cellcolor{yellow}$11989$	& (d)	& $\operatorname{Bose}(109)$	& $\mathbb{Z}_{11880}$\\
$110$	& $12100$	& \cellcolor{yellow}$12101$	& \cellcolor{yellow}$12101$	& (c)	& $\operatorname{Singer}(109)$	& $\mathbb{Z}_{11991}$\\
$111$	& $12321$	& $24864$	& \cellcolor{yellow}$12655$	& (b)	& $\operatorname{Hughes}(113)$	& $\mathbb{Z}_{112}^2$\\
$112$	& $12544$	& \cellcolor{yellow}$12768$	& \cellcolor{yellow}$12768$	& (b)	& $\operatorname{Spence}(113)$	& $\mathbb{Z}_{12656}$\\
$113$	& $12769$	& \cellcolor{yellow}$12881$	& \cellcolor{yellow}$12881$	& (d)	& $\operatorname{Bose}(113)$	& $\mathbb{Z}_{12768}$\\
$114$	& $12996$	& \cellcolor{yellow}$12997$	& \cellcolor{yellow}$12997$	& (c)	& $\operatorname{Singer}(113)$	& $\mathbb{Z}_{12883}$\\
$115$	& $13225$	& $53130$	& \cellcolor{yellow}$14515$	& (b)	& $\operatorname{Hughes}(121)-4$	& $\mathbb{Z}_{120}^2$\\
$116$	& $13456$	& $27144$	& \cellcolor{yellow}$14516$	& (b)	& $\operatorname{Hughes}(121)-3$	& $\mathbb{Z}_{120}^2$\\
$117$	& $13689$	& $109746$	& \cellcolor{yellow}$14517$	& (b)	& $\operatorname{Hughes}(121)-2$	& $\mathbb{Z}_{120}^2$\\
$118$	& $13924$	& $83780$	& \cellcolor{yellow}$14518$	& (b)	& $\operatorname{Hughes}(121)-1$	& $\mathbb{Z}_{120}^2$\\
$119$	& $14161$	& $28560$	& \cellcolor{yellow}$14519$	& (b)	& $\operatorname{Hughes}(121)$	& $\mathbb{Z}_{120}^2$\\
$120$	& $14400$	& \cellcolor{yellow}$14640$	& \cellcolor{yellow}$14640$	& (b)	& $\operatorname{Spence}(121)$	& $\mathbb{Z}_{120}\times \mathbb{Z}_{11}^{2}$\\
$121$	& $14641$	& \cellcolor{yellow}$14761$	& \cellcolor{yellow}$14761$	& (d)	& $\operatorname{Bose}(121)$	& $\mathbb{Z}_{14640}$\\
$122$	& $14884$	& \cellcolor{yellow}$14885$	& \cellcolor{yellow}$14885$	& (c)	& $\operatorname{Singer}(121)$	& $\mathbb{Z}_{14763}$\\
$123$	& $15129$	& $91020$	& \cellcolor{yellow}$15499$	& (b)	& $\operatorname{Hughes}(125)$	& $\mathbb{Z}_{124}^2$\\
$124$	& \cellcolor{yellow}$15376$	& \cellcolor{yellow}$15376$	& $15624$	& (a)	& $\operatorname{Spence}(125)$	& $\mathbb{Z}_{124}\times \mathbb{Z}_{5}^{3}$\\
$125$	& $15625$	& \cellcolor{yellow}$15749$	& \cellcolor{yellow}$15749$	& (d)	& $\operatorname{Bose}(125)$	& $\mathbb{Z}_{15624}$\\
$126$	& $15876$	& \cellcolor{yellow}$15877$	& \cellcolor{yellow}$15877$	& (c)	& $\operatorname{Singer}(125)$	& $\mathbb{Z}_{15751}$\\
$127$	& $16129$	& \cellcolor{yellow}$16255$	& \cellcolor{yellow}$16255$	& (d)	& $\operatorname{Bose}(127)$	& $\mathbb{Z}_{16128}$\\
$128$	& $16384$	& \cellcolor{yellow}$16385$	& \cellcolor{yellow}$16385$	& (c)	& $\operatorname{Singer}(127)$	& $\mathbb{Z}_{16257}$\\
$129$	& $16641$	& \cellcolor{yellow}$16642$	& \cellcolor{yellow}$16642$	& (c)	& $\operatorname{Singer}(128)$	& $\mathbb{Z}_{16513}$\\
$130$	& $16900$	& \cellcolor{yellow}$17160$	& \cellcolor{yellow}$17160$	& (b)	& $\operatorname{Spence}(131)$	& $\mathbb{Z}_{17030}$\\
$131$	& $17161$	& \cellcolor{yellow}$17291$	& \cellcolor{yellow}$17291$	& (d)	& $\operatorname{Bose}(131)$	& $\mathbb{Z}_{17160}$\\
$132$	& $17424$	& \cellcolor{yellow}$17425$	& \cellcolor{yellow}$17425$	& (c)	& $\operatorname{Singer}(131)$	& $\mathbb{Z}_{17293}$\\
$133$	& $17689$	& $177156$	& \cellcolor{yellow}$18629$	& (b)	& $\operatorname{Hughes}(137)-2$	& $\mathbb{Z}_{136}^2$\\
$134$	& $17956$	& $36180$	& \cellcolor{yellow}$18630$	& (b)	& $\operatorname{Hughes}(137)-1$	& $\mathbb{Z}_{136}^2$\\
$135$	& $18225$	& $36720$	& \cellcolor{yellow}$18631$	& (b)	& $\operatorname{Hughes}(137)$	& $\mathbb{Z}_{136}^2$\\
$136$	& $18496$	& \cellcolor{yellow}$18768$	& \cellcolor{yellow}$18768$	& (b)	& $\operatorname{Spence}(137)$	& $\mathbb{Z}_{18632}$\\
$137$	& $18769$	& \cellcolor{yellow}$18905$	& \cellcolor{yellow}$18905$	& (d)	& $\operatorname{Bose}(137)$	& $\mathbb{Z}_{18768}$\\
$138$	& $19044$	& \cellcolor{yellow}$19045$	& \cellcolor{yellow}$19045$	& (c)	& $\operatorname{Singer}(137)$	& $\mathbb{Z}_{18907}$\\
$139$	& $19321$	& \cellcolor{yellow}$19459$	& \cellcolor{yellow}$19459$	& (d)	& $\operatorname{Bose}(139)$	& $\mathbb{Z}_{19320}$\\
$140$	& $19600$	& \cellcolor{yellow}$19601$	& \cellcolor{yellow}$19601$	& (c)	& $\operatorname{Singer}(139)$	& $\mathbb{Z}_{19461}$\\
$141$	& $19881$	& $40044$	& \cellcolor{yellow}$22045$	& (b)	& $\operatorname{Hughes}(149)-6$	& $\mathbb{Z}_{148}^2$\\
$142$	& $20164$	& $80940$	& \cellcolor{yellow}$22046$	& (b)	& $\operatorname{Hughes}(149)-5$	& $\mathbb{Z}_{148}^2$\\
$143$	& $20449$	& $122980$	& \cellcolor{yellow}$22047$	& (b)	& $\operatorname{Hughes}(149)-4$	& $\mathbb{Z}_{148}^2$\\
$144$	& $20736$	& $41760$	& \cellcolor{yellow}$22048$	& (b)	& $\operatorname{Hughes}(149)-3$	& $\mathbb{Z}_{148}^2$\\
$145$	& $21025$	& $210540$	& \cellcolor{yellow}$22049$	& (b)	& $\operatorname{Hughes}(149)-2$	& $\mathbb{Z}_{148}^2$\\
$146$	& $21316$	& $42924$	& \cellcolor{yellow}$22050$	& (b)	& $\operatorname{Hughes}(149)-1$	& $\mathbb{Z}_{148}^2$\\
$147$	& $21609$	& $129948$	& \cellcolor{yellow}$22051$	& (b)	& $\operatorname{Hughes}(149)$	& $\mathbb{Z}_{148}^2$\\
$148$	& $21904$	& \cellcolor{yellow}$22200$	& \cellcolor{yellow}$22200$	& (b)	& $\operatorname{Spence}(149)$	& $\mathbb{Z}_{22052}$\\
$149$	& $22201$	& \cellcolor{yellow}$22349$	& \cellcolor{yellow}$22349$	& (d)	& $\operatorname{Bose}(149)$	& $\mathbb{Z}_{22200}$\\
$150$	& $22500$	& \cellcolor{yellow}$22501$	& \cellcolor{yellow}$22501$	& (c)	& $\operatorname{Singer}(149)$	& $\mathbb{Z}_{22351}$\\\hline	
\end{longtable}
\end{small}

\end{document}